\documentclass{article}
\usepackage{amsfonts, amsthm, amsmath, amssymb, latexsym}

\newtheorem{thm}{Theorem}[section]
\newtheorem{lem}[thm]{Lemma}
\newtheorem{prop}[thm]{Proposition}

\newtheorem{defi}[thm]{Definition}

\newcommand{\ve}{\varepsilon}
\newcommand{\GZ}{\mathbb{Z}}

\newcommand{\GX}{\mathbb{X}}

\newcommand{\la}{\langle}
\newcommand{\ra}{\rangle}

\newfont{\bsl}{cmbxsl10 scaled\magstep3}
\newfont{\bit}{cmti10 scaled\magstep4}
\newfont{\bbit}{cmti10 scaled\magstep5}
\newfont{\bbf}{cmbx10 scaled\magstep4}
\newfont{\bbfb}{cmbx10 scaled\magstep5}
\newfont{\bbbf}{cmbx10 scaled\magstep3}

\begin{document}
\title{{\sf The negative slope algorithm and the dimension group of free rank 3} 
\thanks{}}
\author{{\sc Koshiro} ISHIMURA
\\
}

\date{\empty}

\maketitle
\begin{abstract}

E.~G.~Effros and C-L.~Shen constructed the dimension group of free rank 2 
from the simple continued fraction algorithm.
The notion of negative slope algorithm was introduced by S.~Ferenczi,
C.~Holton, and L.~Zamboni in their study of 3-interval exchange transformations. 
The negative slope algorithm is the 2-dimensional continued fraction algorithm.
Then the author succeed to construct the dimensional group of free rank 3 by the similar method
which E.~G.~Effros and C-L.~Shen used.
\end{abstract}

\section{Introduction}
\(\)\\[-5mm]
\indent E.~G.~Effros and C-L.~Shen \cite{ES} showed that the continued fraction expansion of \(\alpha\) which is irrational explicitly
  determines approximately finite algebras \(A\) with the dimension group \(G(A) \cong_{\mbox{ord}} (\GZ^2, P_{\alpha})\) and how one
   can use Bratteli diagrams to determine dicylic dimenstion groups. \\
   \indent
 The negative slope algorithm(NSA) was introduced by S.~Ferenczi, C.~Holton and L.~Zamboni \cite{FHZ1}, \cite{FHZ2}, \cite{FHZ3}
 to discuss the structure of 3-interval exchange transformations. NSA is a kind of 2-dimensional continued
 fraction algorithms. They discussed some arithmetic properties, the natural codings of 3-interval exchange 
 transformations and show the necessary sufficient condition which 3-interval exchange transformations are weak mixing.
 in those papers. In \cite{IN}, the author and H.~Nakada showed that NSA is weak Bernoulli by using Yuri's condition by \cite{Y}.
 They also introduced the natural extension of NSA to calculate the entropy of NSA and drive the absolutely continuous invariant 
 measure for NSA as the marginal distribution.
 The author and S.~Ito showed the necessary sufficient condition for the orbit of NSA being purely periodic by
 using the natural extension of NSA.\\
 \indent In this paper, we give the definition of the dimension group and NSA in \S 2 and \S 3 respectively.
 In \S 4, we construct the dimension group of free rank 3 for given \( (\alpha, \beta) \in [0,1)^2 \backslash \{ (x,y) \ | \ x+y=1 \} \)
 whose iteration does not stop by NSA.


\section{Definitions of the dimension group}
\indent At first, we prepare some definitions to define the dimension group.
\noindent \begin{defi}{\bf Positive cone \(G_{+}\) of \(G\) }\\
\indent Let \(G\) be an abelian group. For a subset \(S\) of \(G\), 
we denote \( -S = \{ -g \ | \ g \in S \} \).
Then, a subset \(G_{+}\) of \(G\) is the positive cone, 
if \(G_{+}\) satisfies the following.
\begin{eqnarray*}
\begin{array}{cl}
(1) & G_{+} \mbox{ is semi-group, that is, } G_{+} + G_{+} \subset G_{+} \\
(2) & G_{+} \cap (-G_{+}) = \{ 0 \} \\
(3) & G_{+} - G_{+} = G \\
(4) & n : \mbox{positive integer, } ng \in G_{+} \Rightarrow g \in G_{+} 
\end{array}
\end{eqnarray*}
\end{defi}

\indent
If \(G\) has the positive cone \(G_{+}\) in the above sense, \(G\) is torsion-free, 
that is, the following is hold;
\begin{eqnarray*}
&& \mbox{For} \ n : \mbox{non-zero positive integer}, \\
&& \quad ng=0 \Rightarrow g=0.
\end{eqnarray*}

\noindent \begin{defi}{{\bf Partial order}}\\
\indent For \(g, h \in G\), we define
\[
g \ge h \Longleftrightarrow g-h \in G_{+}.
\]
Then, it is easy to see that \( \ge \) is the partial order on \(G\).
\end{defi}
 
\noindent 
{\bf Example}\\
\indent 
Let \(G = \GZ^n\) and \( G_{+} = (\GZ_{+})^n\), then \( G_{+}\) is the positive cone of \(G\) 
where \( \GZ_{+} \) is the set of all non-negative integers. 

 \indent We call \( (G, G_{+})\) as the simplicial group of free rank \(n\). \\
 \indent
 For example, let \(A\) is a \( m \times n \) integer matrix, then 
\begin{eqnarray*}
A \left( (\GZ_{+})^n \right) \subset (\GZ_{+})^m  \Longleftrightarrow \ \mbox{ the each component of \(A\) is non-negative }.
\end{eqnarray*}

\begin{defi}{\bf The positive induction sequence}\\
\indent 
We call \( \{ A_{k+1} \ : \ \GZ^{n_k} \rightarrow \GZ^{n_{k+1}} \}_{k \ge 0} \) 
as a positive induction sequence 
where \( A_{k+1}\) is a \( n_{k+1} \times n_k \) non-negative integer matrix.
\end{defi}
\indent 
For simplicity, we call this positive induction sequence as the induction.\\

\begin{defi}{\bf The inductive limit group of the induction}\\
\indent
Let \( \{ A_{k+1} : \GZ^{n_k} \rightarrow \GZ^{n_{k+1}} \}_{k \ge 0} \) 
is the induction.  We put a direct sum group as follows:
\begin{eqnarray*}
&& \bigoplus_{k \ge 0} \GZ^{n_k} := \left\{ (g_0, g_1, g_2, \cdots) \ | \ g_k \in \GZ^{n_k} \mbox{ and } \sharp \{ k \ge 0 \ | \ g_k \neq 0 \} < \infty \right\} \\
&& \mbox{where, for} \ \Vec{g}=(g_0, g_1, g_2, \cdots), \ \Vec{h}=(h_0, h_1, h_2, \cdots) \in \bigoplus_{k \ge 0} \GZ^{n_k}, \\
&& \mbox{we put} \ \Vec{g} + \Vec{h} = (g_0, g_1, g_2, \cdots)+(h_0, h_1, h_2, \cdots)=(g_0+h_0, g_1+h_1, g_2+h_2, \cdots) .
\end{eqnarray*}
Then, for each \( k \ge 0 \), we obtain the following natural embedding;
\[
j_k : \GZ^{n_k} \hookrightarrow \bigoplus_{k \ge 0} \GZ^{n_k} : g \mapsto (0,0, \cdots, 0, \ \underset{k~th}g \ , 0,0, \cdots ) .
\]
\end{defi}

Next, we consider the following subgroup of the direct sum group;
\begin{eqnarray*}
A & = & \mbox{ the subgroup which generated by } \\
& & \qquad \{ j_l \circ A_l A_{l-1} \cdots A_{k+1} (g) - j_k (g) \ | \ 0 \leq k < l, g \in \GZ^{n_k} \} \\
& := &  \left\{ \sum_{\mbox{\tiny  finite sum  }} j_{l_i} \circ A_{l_i} A_{l_i-1} \cdots A_{k_i +1} (g_i) - j_{k_i} (g_i) \right. \\
&& \qquad \qquad  \qquad \qquad \qquad \qquad \qquad \left. | \ 0 \leq k_i < l_i, g_i \in \GZ^{n_{k_i}} \right\}  \\[5mm]
& \subset &  \bigoplus_{k \ge 0} \GZ^{n_k}  .
\end{eqnarray*}
\indent
We denote the quotient group of the direct sum by \(A\) as 
\[
\bigoplus_{k \ge 0} \GZ^{n_k} / A =: \varinjlim ( \GZ^{n_k} , A_{k+1})
\]
and call it the inductive limit group of the induction \( \{ A_{k+1} : \GZ^{n_k} \rightarrow \GZ^{n_{k+1}} \}_{k \ge 0} \).
We denote the composition of the embedding \( j_k \) and the projection 
\( q : \bigoplus_{k \ge 0} \GZ^{n_k} \rightarrow 
\varinjlim ( \GZ^{n_k} , A_{k+1}) \) as 
\[
\theta_k = q \circ j_k \ : \ \GZ^{n_k} \rightarrow \varinjlim ( \GZ^{n_k} , A_{k+1}),
\]
Then we call it the canonical homomorphism.\\
\indent
Then we have the following proposition.\\
\begin{prop}
\[ \theta_{k} = \theta_{k+1} \circ A_{k+1} \]
\end{prop}
\begin{proof}
\indent
Let \( g \in \GZ^{n_k} \). Then, since \( j_{k+1} A_{k+1}(g) - j_k (g) \in A\), we have
\begin{eqnarray*}
 \theta_{k+1} \circ A_{k+1}(g) - \theta_k (g) & = & q \circ j_{k+1} A_{k+1} (g) - q \circ j_k (g) \\
 & = & q \left( j_{k+1} A_{k+1} (g) - j_k (g) \right) = 0.
\end{eqnarray*}

\end{proof}
Now we define the dimension group determined by the induction.

\begin{defi}{\bf The dimension group determined by the induction}\\
\indent
Let \( \{ A_{k+1} : \GZ^{n_k} \rightarrow \GZ^{n_{k+1}} \}_{k \ge 0} \) be the induction. 
Then the inductive limit group \( G := \varinjlim ( \GZ^{n_k} , A_{k+1}) \)  of the induction \( \{ A_{k+1} \}_{k \ge 0} \) has 
the following natural positive cone;
\[ G_{+} := \bigcup_{k \ge 0} \theta_k ( (\GZ_{+})^{n_k} ) .
\]
Then, we call \( (G, G_{+}) \) as the dimension group determined by the induction \( \{ A_{k+1} \}_{k \ge 0} \).
\end{defi}


\section{Definitions and some properties of the negative slope algorithm}
\subsection{Definitions of the negative slope algorithm}
\quad First we introduce a map $T$ which is called the negative slope algorithm on the unit square in \( [0,1]\).  Let \(\GX = [0, \, 1]^2 \setminus \{(x, y)\, | \, 
x+y =1 \}$, we define a map \(T\) on \(\GX\) by
\[
T(x, y) \, = \, \left\{ 
\begin{array}{ccc}
\left( \frac{y}{(x+y) - 1} - \left[\frac{y}{(x+y) - 1} \right], \, 
\frac{x}{(x+y) - 1} - \left[\frac{x}{(x+y) - 1} \right] \right) & 
 \mbox{if} & x + y \, > \, 1 \\
 {}&{}& {} \\
\left( \frac{1-y}{1 - (x+y) } - \left[\frac{1-y}{1-(x+y)} \right], \, 
\frac{1-x}{1-(x+y)} - \left[\frac{1-x}{1-(x+y)} \right] \right) & 
 \mbox{if} & x + y \, < \, 1  .
\end{array} \right.
\]

Using the integer valued functions 
\[
(n(x, y), m(x, y)) \, = \, \left\{ 
\begin{array}{ccc}
\left( \left[\frac{y}{(x+y) - 1} \right]  ,  \left[\frac{x}{(x+y) - 1} \right] \right)
 & 
 \mbox{if} & x + y \, > \, 1 \\
 {} & {} & {} \\
\left( \left[\frac{1-y}{1-(x+y)} \right],  \left[\frac{1-x}{1-(x+y)} \right] \right) & 
 \mbox{if} & x + y \, < \, 1 ,
\end{array} \right.
\]
and 
\[
\varepsilon(x, y) \, = \, \left\{ 
\begin{array}{ccl}
-1 & \mbox{if} & x + y \, > \, 1 \\
 +1  &  \mbox{if} & x + y \, < \, 1 ,
\end{array} \right.
\]
for each $(x, y) \in \GX$, we have a sequence 
\[
\left( (\ve_1(x, y), n_1(x, y), m_1(x, y)), \, 
(\ve_2(x, y), n_2(x, y), m_2(x, y)), \, 
\ldots, \right).
\]
We obtain it by
\[
\left\{
\begin{array}{ccc}
\varepsilon_k (x, y) & = & \varepsilon( T^{k-1}(x, y)) \\
n_k (x, y) & = & n( T^{k-1}(x, y)) \\
m_k (x, y) & = & m( T^{k-1}(x, y)) 
\end{array} \right.
\]
for $k \ge 1$. 
\begin{lem}{\rm (\cite{II}, Lemma 2.5)} For $n_i, m_i \ge 1$, $ i\ge 1 $ 
and for any sequence $\left( (\ve_i, n_i, m_i), 
i \ge 1\right)$, there exists $(x, y) \in \GX$ such that 
$(\ve_i(x, y), n_i(x, y), m_i(x, y)) \, = \, (\ve_i, n_i, m_i) $ 
unless there exists $k \ge 1$ such that $ ( \ve_i, m_i) = (+1, 1) $ for $ i \ge k $ 
or $ ( \ve_i, n_i) = (+1, 1) $ for $ i \ge k $. 
\end{lem}
By \cite{FHZ1} and \cite{IN} 
we see that if $(x, y) \ne (x', y') \in \GX$, then there exists $k \ge 1$ such that 
\begin{equation*}
(\ve_k(x, y), n_k(x, y), m_k(x, y)) \, \ne \, 
(\ve_k(x', y'), n_k(x', y'), m_k(x', y')) . 
\end{equation*}


Next we introduce a projective representation of $T$ as follows.  We put 
\[
A_{(+1,n,m)} \, = \, 
\begin{pmatrix}
n   &  n-1  &  1-n  \\
m-1 &   m   &  1-m  \\
-1   &   -1   &  1 
\end{pmatrix}
\]
and 
\[
A_{(-1,n,m)} \, = \, 
\begin{pmatrix}
-n   &  -n+1  &  n  \\
-m+1 &   -m   &  m  \\
1    &    1   &  -1 
\end{pmatrix}
\]
for $m, n \ge 1$. Then we have 
\[
A_{(+1,n,m)}^{-1} \, = \, 
\begin{pmatrix}
  1   &   0  &   n-1  \\
  0   &   1  &   m-1  \\
  1   &   1  &  n+m-1 
\end{pmatrix} 
\]
and 
\[
A_{(-1,n,m)}^{-1} \, = \, 
\begin{pmatrix}
  0   &   1  &   m  \\
  1   &   0  &   n  \\
  1   &   1  &  n+m-1 
\end{pmatrix} .
\]
We identify $(x,y) \in \GX$ to $\begin{pmatrix} \alpha x \\ \alpha y \\
\alpha \end{pmatrix}$ for $\alpha \ne 0$.  Then we identify $T(x,y)$ to 
\[
A_{(\ve_1(x, y), n_1(x, y), m_1(x, y))} \begin{pmatrix} 
x \\ y \\ 1 \end{pmatrix} 
\]
and its local inverse is given by 
\[
A_{(\ve_1(x, y), n_1(x, y), m_1(x, y))}^{-1} .
\]
In this way, we get a representation of $(x, y) \in \GX$ by 
\[
A_{(\ve_1, n_1, m_1)}^{-1} \, A_{(\ve_2, n_2, m_2)}^{-1} \, 
A_{(\ve_3, n_3, m_3)}^{-1} \, \cdots 
\]
and $T$ is defined as a multiplication by $A_{(\ve_1, n_1, m_1)}$ 
from the left and acts as 
a shift on the set of infinite sequences of matrices 
\[
\left\{ 
A_{(\ve_1, n_1, m_1)}^{-1} \, A_{(\ve_2, n_2, m_2)}^{-1} \, 
A_{(\ve_3, n_3, m_3)}^{-1} \, \cdots \, | \, 
\ve_k = \pm 1, \, n_k, m_k \ge 1 \, \mbox{for} \, k \ge1
\right\} .
\]
For a given finite sequence $ \left( (\ve_1, n_1, m_1), \, (\ve_2, n_2, m_2), \, 
\ldots , \, (\ve_k, n_k, m_k) \right) $, we define a cylinder set of length 
$k$ by 
\begin{eqnarray*}
\lefteqn{
\la (\ve_1, n_1, m_1), \, (\ve_2, n_2, m_2), \, 
\ldots , \, (\ve_k, n_k, m_k) \ra } \\ 
&& = \, 
\{ (x, y) \in \GX \, | \, (\ve_i(x, y), n_i(x, y), m_i(x, y)) = (\ve_i, n_i, m_i) 
 , \, 1 \le i \le k \} .
\end{eqnarray*}
For simplicity, we write \( \Delta_k \) for this cylinder set.\\

\indent
For $(x, y) \in \Delta_k$, we denote $T^k(x, y)$ as 
\[
A_{(\ve_k, n_k, m_k)} \cdots A_{(\ve_1, n_1, m_1)} 
\begin{pmatrix} x \\ y \\ 1 \end{pmatrix}
\]
and its local inverse $\Psi_{\Delta_k}$ as 
\[
A_{(\ve_1, n_1, m_1)}^{-1} \cdots A_{(\ve_k, n_k, m_k)}^{-1} .
\]
We put 
\[
\Psi_{\Delta_k} = A_{(\ve_1, n_1, m_1)}^{-1} \cdots A_{(\ve_k, n_k, m_k)}^{-1} 
\, = \, 
\begin{pmatrix}
p_1^{(k)}  &  p_2^{(k)}  &  p_3^{(k)} \\
r_1^{(k)}  &  r_2^{(k)}  &  r_3^{(k)} \\
q_1^{(k)}  &  q_2^{(k)}  &  q_3^{(k)} 
\end{pmatrix}
\]
for any sequence $\left( (\ve_1, n_1, m_1), \, (\ve_2, n_2, m_2), \, 
\ldots , \, (\ve_k, n_k, m_k) \right) $, $k \ge 1$. \\
\indent Since 
\begin{eqnarray*}
{} &{} & \left\{ \left( \frac{y}{(x+y) - 1}, \frac{x}{(x+y) - 1} \right) \, : \, (x, y) 
\in \GX, \,    x+y > 1 \right\} \\
& = & 
\left\{ \left( \frac{1-y}{1 -(x+y)}, \frac{1-x}{1-(x+y)} \right) \, : \, (x, y)  \in \GX, \,  x+y < 1 \right\} \\[0.3cm]
& = & \{ (x', y') \, : \, x' \ge 1, \, y' \ge 1 \},
\end{eqnarray*}
we see that 
\begin{equation*}
T^{j} \{ (x, y) \in \GX \, : \, \ve_k(x, y) = \ve_k, \,  n_k(x,y) = n_k, \, 
m_k(x,y) = m_k, \, 1 \le k \le j \} \, = \, \GX 
\end{equation*}
for any $\{ (\ve_k, n_k, m_k), \, 1 \le k \le j \}$, 
$\ve_k = \pm 1 $, $n_k, m_k \ge 1$ without the boundary of \(\GX\).\\

\subsection{The case where the negative slope algorithm stops}
Next we define what means that the iteration by the negative slope algorithm 
\(T\) of \((x,y)\ \in \GX\) stops.
\begin{defi}We denote k-th iteration by the negative slope algorithm 
\(T\) of \((x,y) \in \GX\) as \((x_k,y_k) = T^k (x,y)\). 
Then we say iteration by the negative slope algorithm \(T\) of \((x,y) \in \GX\) stops 
if there exists \(k_0 \geq 0\) such that \(x_{k_0} =0 \) or \(y_{k_0} =0\) or \( x_{k_0} + y_{k_0} =1 \).
\end{defi}
This implies that iteration by the negative slope algorithm 
\(T\) of \((x,y) \in \GX\) stops if there exists \(k_0 \ge 0 \) s.t. 
\((x_{k_0}, y_{k_0}) \in \partial \GX\).
From this definition, we get the following propositions.
\begin{prop}{\rm \cite{II}} \ \ 
If iteration by the negative slope algorithm \(T\) of \((x,y) \in \GX\) stops, then \((x,y)\) satisfies one of the following equations.
\begin{eqnarray*}
&(p+1)x+py=q& \\
&px+(p+1)y =q& \\
&px+py=q& 
\end{eqnarray*}
for some integers \( 0 \leq q \leq 2p\).
\end{prop}
\begin{prop}{\rm \cite{II}} \ \ 
If \((x,y) \in \GX\) satisfies the following equation,
\[ px+py =q\]
for any integers \( 0 \le q \le 2p\), then there exists \(N>0\) such that the sequence \((T^k(x,y) : k \ge 0)\) terminates at \(k=N\) for the negative slope algorithm \(T\). 
\end{prop}

\subsection{Properties of the negative slope algorithm}
\indent
From \S 3.1, it is easy to see that \(p^{(k)}_i\) and \(r^{(k)}_i\) are non-negative integers 
and \( q_i^{(k)} \) is positive integer for \( i=1,2,3\), \(k \ge 0\). 
In this subsection, we show some properties for entries of \( \Psi_{\Delta_k} \).

\begin{lem}{\rm \cite{II}} \ 
For the entries of \(\Psi_{\Delta_k}\), we have
\begin{eqnarray*}
\left\{ \begin{array}{ccl}
p_1^{(k)} &=& p_2^{(k)} + \ve_1 \cdots \ve_k \\
r_1^{(k)} &=& r_2^{(k)} - \ve_1 \cdots \ve_k \\
q_1^{(k)} &=& q_2^{(k)} \qquad \quad \quad \quad .
\end{array}
\right. 
\end{eqnarray*}
\end{lem}

\begin{lem}
For \((x ,y) \in \GX\), we have
\begin{eqnarray*}
 \frac{p^{(k)}_1+r^{(k)}_1}{q^{(k)}_1}= \frac{p^{(k)}_2+r^{(k)}_2}{q^{(k)}_2}= \frac{p^{(k)}_3+r^{(k)}_3}{q^{(k)}_3} +  \frac{\delta_k}{ q^{(k)}_2 q^{(k)}_3 }
\end{eqnarray*}
where \(\delta_k = \ve_1(x ,y) \cdots \ve_k(x ,y)\).
\end{lem}

\begin{proof}
By taking a determinant of \(\Psi_{\Delta_k}\), we have
\begin{eqnarray*}
\left| \begin{array}{ccc}
p_1^{(k)}  &  p_2^{(k)}  &  p_3^{(k)} \\
r_1^{(k)}  &  r_2^{(k)}  &  r_3^{(k)} \\
q_1^{(k)}  &  q_2^{(k)}  &  q_3^{(k)} 
\end{array} \right|
= p_1^{(k)} \left| \begin{array}{cc}
  r_2^{(k)}  &  r_3^{(k)} \\
  q_2^{(k)}  &  q_3^{(k)} 
\end{array}\right| - p_2^{(k)}  \left| \begin{array}{cc}
r_1^{(k)}    &  r_3^{(k)} \\
q_1^{(k)}   &  q_3^{(k)} 
\end{array} \right| + p_3^{(k)} \left| \begin{array}{cc}
r_1^{(k)}  &  r_2^{(k)}  \\
q_1^{(k)}  &  q_2^{(k)}  
\end{array} \right|.
\end{eqnarray*}
From Lemma 3.1, the right hand side is equal to
\begin{eqnarray*}
(p_2^{(k)} + \delta_k) \left| \begin{array}{cc}
  r_2^{(k)}  &  r_3^{(k)} \\
  q_2^{(k)}  &  q_3^{(k)} 
\end{array}\right| - p_2^{(k)}  \left| \begin{array}{cc}
r_2^{(k)} - \delta_k   &  r_3^{(k)} \\
q_1^{(k)}   &  q_3^{(k)} 
\end{array} \right| + p_3^{(k)} \left| \begin{array}{cc}
r_2^{(k)} -\delta_k  &  r_2^{(k)}  \\
q_1^{(k)}  &  q_2^{(k)}  
\end{array} \right|
\end{eqnarray*}
where \(\delta_k = \ve_1(x ,y) \cdots \ve_k(x ,y)\). Since \( \det \Psi_{\Delta_k} =1 \), we have
\begin{eqnarray} ( r_2^{(k)} q_3^{(k)} - r_3^{(k)} q_2^{(k)} ) + ( p_2^{(k)} q_3^{(k)} - p_3^{(k)} q_2^{(k)} ) = \delta_k. 
\end{eqnarray}
Substituting \( p_1^{(k)} = p_2^{(k)} + \delta_k \), \( r_1^{(k)} = r_2^{(k)} - \delta_k  \) and \(q_1^{(k)}=q_2^{(k)}\) for \((1)\), we see that
\begin{eqnarray} ( r_1^{(k)} q_3^{(k)} - r_3^{(k)} q_1^{(k)} ) + ( p_1^{(k)} q_3^{(k)} - p_3^{(k)} q_1^{(k)} ) = \delta_k. 
\end{eqnarray}
From \((1)\) and \((2)\), we have
\begin{eqnarray}
\frac{p_1^{(k)}+r_1^{(k)}}{q_1^{(k)}} = \frac{p_2^{(k)}+r_2^{(k)}}{q_2^{(k)}}= \frac{p_3^{(k)}+r_3^{(k)}}{q_3^{(k)}} + \frac{\delta_k}{q_2^{(k)} q_3^{(k)}}.
\end{eqnarray}

\end{proof}

\begin{lem} For \(i=1,2,3\), we have 
\[
\lim_{k \rightarrow \infty} \left( \frac{p_i^{(k)}}{q_i^{(k)}}, \frac{r_i^{(k)}}{q_i^{(k)}} \right) = ( \alpha, \beta) .
\]
\end{lem}
\begin{proof}
From theorem 4.1 of \cite{IN}, we see that 
\[
\lim_{k \rightarrow \infty} \left( \frac{p^{(k)}_3}{q^{(k)}_3}, \frac{r^{(k)}_3}{q^{(k)}_3} \right) = (\alpha, \beta).
\]

\indent Since 
\[
A_{(\ve_1, n_1, m_1)}^{-1} \cdots A_{(\ve_k, n_k, m_k)}^{-1} 
\, = \, \left\{ \begin{array}{r} 
\begin{pmatrix}
p_1^{(k-1)}  &  p_2^{(k-1)}  &  p_3^{(k-1)} \\
r_1^{(k-1)}  &  r_2^{(k-1)}  &  r_3^{(k-1)} \\
q_1^{(k-1)}  &  q_2^{(k-1)}  &  q_3^{(k-1)} 
\end{pmatrix}
\begin{pmatrix}
1   &   0   &   n_k  - 1    \\
0   &   1   &   m_k -1   \\
1   &   1   &   n_k + m_k -1 
\end{pmatrix}
 \mbox{     if} \ \ve_k \, = \, +1  \\[7mm]
\begin{pmatrix}
p_1^{(k-1)}  &  p_2^{(k-1)}  &  p_3^{(k-1)} \\
r_1^{(k-1)}  &  r_2^{(k-1)}  &  r_3^{(k-1)} \\
q_1^{(k-1)}  &  q_2^{(k-1)}  &  q_3^{(k-1)} 
\end{pmatrix}
\begin{pmatrix}
0   &   1   &   m_k    \\
1   &   0   &   n_k    \\
1   &   1   &   n_k + m_k -1 
\end{pmatrix} 
 \mbox{     if} \ \ve_k \, = \, -1
\end{array} \right. , 
\]
we see that
\begin{eqnarray*} 
(p_1^{(k)}, p_2^{(k)})=
\left\{
\begin{array}{lll}
(p_1^{(k-1)}+p_3^{(k-1)}, p_2^{(k-1)}+p_3^{(k-1)}) & \mbox{if} & \ve = +1 \\[4mm]
(p_2^{(k-1)}+p_3^{(k-1)}, p_1^{(k-1)}+p_3^{(k-1)}) & \mbox{if} &\ve = -1 
\end{array}
\right.
\end{eqnarray*}
and
\begin{eqnarray*}
(r_1^{(k)}, r_2^{(k)})=
\left\{
\begin{array}{lll}
(r_1^{(k-1)}+r_3^{(k-1)}, r_2^{(k-1)}+r_3^{(k-1)}) & \mbox{if} & \ve = +1 \\[4mm]
(r_2^{(k-1)}+r_3^{(k-1)}, r_1^{(k-1)}+r_3^{(k-1)}) & \mbox{if} & \ve = -1 
\end{array}
\right. .
\end{eqnarray*}
Then we have
\begin{eqnarray*}
\left( \frac{p_1^{(k)}}{q_1^{(k)}}, \frac{r_1^{(k)}}{q_1^{(k)}} \right) &=&
\left\{
\begin{array}{lll}
\left( {\displaystyle \frac{p_1^{(k-1)}+p_3^{(k-1)}}{q_1^{(k-1)}+q_3^{(k-1)}}, \frac{r_1^{(k-1)}+p_3^{(k-1)}}{q_1^{(k-1)}+q_3^{(k-1)}}} \right) & \mbox{if} & \ve = +1 \\[7mm]
\left( {\displaystyle \frac{p_2^{(k-1)}+p_3^{(k-1)}}{q_1^{(k-1)}+q_3^{(k-1)}}, \frac{r_2^{(k-1)}+p_3^{(k-1)}}{q_2^{(k-1)}+q_3^{(k-1)}}} \right) & \mbox{if} & \ve = -1 
\end{array}
\right. \\[3mm]
&=&
\left\{
\begin{array}{lll}
\Psi_{\Delta_{k-1}} (0,1) & \mbox{if} & \ve = +1 \\[4mm]
\Psi_{\Delta_{k-1}} (1,0) & \mbox{if} & \ve = -1 
\end{array}
\right. .
\end{eqnarray*}
From lemma 3.5, we have
\[ \left( \frac{p_1^{(k)}}{q_1^{(k)}}, \frac{r_1^{(k)}}{q_1^{(k)}} \right) = \left( \frac{p_2^{(k)}}{q_2^{(k)}}, \frac{r_2^{(k)}}{q_2^{(k)}} \right). \]
From theorem 4.1 of \cite{IN}, this implies that
\[
\lim_{k \rightarrow \infty} \left( \frac{p_i^{(k)}}{q_i^{(k)}}, \frac{r_i^{(k)}}{q_i^{(k)}} \right) = (\alpha, \beta) \quad (i=1,2,3)
\]
\end{proof}


\section{A construction of the dimension group of free rank 3}

\indent In this section, we proof our main theorem.
\begin{thm}
Assume that \( (\alpha ,\beta)  \in \GX \) does not stop by the negative slope algorithm. 
We put \( G\) and \( G_{+} \) as follows:
\begin{eqnarray*}
&& G = \varinjlim \left( \GZ^3, ~^{t}A^{-1}_{ (\ve_{k+1}, n_{k+1}, m_{k+1}) } \right), \\[2mm]
&& G_{+} = P_{\alpha, \beta} := \{ {\bf 0}  \} \cup \{ {\bf v} \in \GZ^3 | (\alpha, \beta, 1) {\bf v} >0  \}.
\end{eqnarray*}
For the induction \( \{ ^{t}A^{-1}_{ (\ve_{k+1}, n_{k+1}, m_{k+1}) } : \GZ^3 \rightarrow \GZ^3 \}_{k \ge 0} \), 
we put \( \{ \theta_k : \GZ^3 \rightarrow G \}_{k \ge 0} \) as follows:
\begin{eqnarray*}
&&\theta_k = \left\{
\begin{array}{ll}
identity & k=0 \\[3mm]
^{t} \Psi_{\Delta_k}^{-1} & k \ge 1.
\end{array}
\right. 
\end{eqnarray*}
\indent Then \( ( G, G_{+}) \) is the dimension group determined by the induction 
\( \{ ^{t}A^{-1}_{ (\ve_{k+1}, n_{k+1}, m_{k+1}) } \}_{k \ge 0}\), 
which has \( \{\theta_k \}_{k \ge 0} \) as the canonical homomorphism.
\end{thm}

\begin{proof}
\indent
We see that \( \theta_{k} = \theta_{k+1} \circ~^{t}A^{-1}_{ (\ve_{k+1}, n_{k+1}, m_{k+1}) } \) and \( \theta_{k} \) is isomorphic.
Then \( G \) is the inductive limit group of the induction \( \{ ^{t}A^{-1}_{ (\ve_{k+1}, n_{k+1}, m_{k+1}) } \}_{k \ge 0} \).
Therefore it is enough for us to show the following equation for the positive cone \( P_{\alpha, \beta} \).
\[ 
P_{\alpha, \beta} = \bigcup_{k \ge 0} \theta_k \left( ( \GZ_{+} )^3 \right)
\]
\noindent
\( (\subset) \) \ Assume \( {\bf v} = ~ ^{t}(v_1, v_2, v_3) \in P_{\alpha, \beta} \backslash \{ {\bf 0} \} \), then we have
\begin{eqnarray}
\theta_k^{-1} ( {\bf v} ) & =& ^{t} \Psi_{\Delta_k} ( {\bf v}) = 
\begin{pmatrix}
p_1^{(k)} & r_1^{(k)} & q_1^{(k)} \\[1.5mm]
p_2^{(k)} & r_2^{(k)} & q_2^{(k)} \\[1.5mm]
p_3^{(k)} & r_3^{(k)} & q_3^{(k)} 
\end{pmatrix}
\begin{pmatrix}
v_1 \\ v_2 \\ v_3 
\end{pmatrix} \nonumber \\[1.5mm]
&=& 
\begin{pmatrix}
p_1^{(k)}v_1+ r_1^{(k)} v_2 + q_1^{(k)} v_3 \\[1.5mm]
p_2^{(k)}v_1+ r_2^{(k)} v_2 + q_2^{(k)} v_3 \\[1.5mm]
p_3^{(k)}v_1+ r_3^{(k)} v_3 + q_3^{(k)} v_3
\end{pmatrix} \nonumber
\end{eqnarray}

\begin{eqnarray*}
=
\begin{pmatrix}
q_1^{(k)} \left( \frac{p_1^{(k)}}{q_1^{(k)}} v_1+ \frac{r_1^{(k)}}{q_1^{(k)}} v_2 +  v_3 \right) \\[3mm]
q_2^{(k)} \left( \frac{p_2^{(k)}}{q_2^{(k)}} v_1+ \frac{r_2^{(k)}}{q_2^{(k)}} v_2 +  v_3 \right) \\[3mm]
q_3^{(k)} \left( \frac{p_3^{(k)}}{q_3^{(k)}} v_1+ \frac{r_3^{(k)}}{q_3^{(k)}} v_2 +  v_3 \right)
\end{pmatrix}.
\end{eqnarray*}

Then, for \( i=1,2,3 \), we obtain
\begin{eqnarray*}
\lim_{k \rightarrow \infty} \frac{p_i^{(k)}}{q_i^{(k)}} v_1 + \frac{r_i^{(k)}}{q_i^{(k)}} v_2 + v_3 
 =   \alpha v_1 + \beta v_2 + v_3 =   ( \alpha, \beta, 1) {\bf v} > 0.
\end{eqnarray*}

For enough large \(k \ge 1\), we see that
\[
\theta^{-1}_{k} ( \bf{v}) \in (\GZ)^3.
\]

Therefore, we have
\[ {\bf v} \in \bigcup_{k \ge 0} \theta_k \left( ( \GZ_{+} )^3 \right).
\]
\noindent
\( (\supset)  \) \ Assume \( {\bf v}= ~ ^{t}(v_1, v_2, v_3) \in \bigcup_{k \ge 0} \theta_k \left( ( \GZ_{+} )^3 \right) \). 
Then, there exists \( a, b, c \in \GZ_{+} \cup \{ {\bf 0} \} \) such that
\[
{\bf v} = \theta_k ( a \textbf{e}_1 + b \textbf{e}_2 + c \textbf{e}_3 ) 
= a \theta_k ( \textbf{e}_1) + b \theta ( \textbf{e}_2 ) + c \theta_k (\textbf{e}_3)
\]
where \( \textbf{e}_1 =~^t(1,0,0), \textbf{e}_2 =~^t(0,1,0), \textbf{e}_3 =~^t(0,0,1)\). 
Therefore, we enough to show the following:
\[
\theta_k ( \textbf{e}_i ) \in P_{\alpha, \beta} \quad (i =1,2,3).
\]
Let \( (\alpha_k, \beta_k ) = T^k (\alpha, \beta)\) for \( k \ge 1 \). We have
\begin{eqnarray*}
C \begin{pmatrix} \alpha_k \\ \beta_k \\ 1 \end{pmatrix} = \Psi_{\Delta_k}^{-1} \begin{pmatrix} \alpha \\ \beta \\ 1 \end{pmatrix}
\end{eqnarray*}
for some \(C \neq 0 \). By taking the cofactor matrix of \( \Psi_{\Delta_k} \), the inverse of \( \Psi_{\Delta_k} \) is equal to
\begin{eqnarray*}
\begin{pmatrix}
\left| \begin{array}{cc} r_2^{(k)} & r_3^{(k)} \\ q_2^{(k)} & q_3^{(k)} \end{array} \right| & 
-\left| \begin{array}{cc} p_2^{(k)} & p_3^{(k)} \\ q_2^{(k)} & q_3^{(k)}\end{array} \right| & 
\left| \begin{array}{cc} p_2^{(k)} & p_3^{(k)} \\ r_2^{(k)} & r_3^{(k)} \end{array} \right| \\[0.5cm] 
-\left| \begin{array}{cc} r_1^{(k)} & r_3^{(k)} \\ q_1^{(k)} & q_3^{(k)} \end{array} \right| & 
\left| \begin{array}{cc} p_1^{(k)} & p_3^{(k)} \\ q_1^{(k)} & q_3^{(k)} \end{array} \right| & 
-\left| \begin{array}{cc} p_1^{(k)} & p_3^{(k)} \\ r_1^{(k)} & r_3^{(k)} \end{array} \right| \\[0.5cm] 
\left| \begin{array}{cc} r_1^{(k)} & r_2^{(k)} \\ q_1^{(k)} & q_2^{(k)} \end{array} \right| & 
-\left| \begin{array}{cc} p_1^{(k)} & p_2^{(k)} \\ q_1^{(k)} & q_2^{(k)} \end{array} \right| & 
\left| \begin{array}{cc} p_1^{(k)} & p_2^{(k)} \\ r_1^{(k)} & r_2^{(k)} \end{array} \right| 
\end{pmatrix}.
\end{eqnarray*}
Then we have
\begin{eqnarray}
\alpha_k = \frac{(r_2^{(k)} q_3^{(k)}-r_3^{(k)} q_2^{(k)})\alpha  +(- p_2^{(k)} q_3^{(k)} + p_3^{(k)} q_2^{(k)})\beta + (p_2^{(k)} r_3^{(k)}-p_3^{(k)} r_2^{(k)})}{(r_1^{(k)} q_2^{(k)}-r_2^{(k)} q_1^{(k)})\alpha + ( - p_1^{(k)} q_2^{(k)} + p_2^{(k)} q_1^{(k)})\beta + ( p_1^{(k)} r_2^{(k)}-p_2^{(k)} r_1^{(k)})},\nonumber \\[0.5cm]
\beta_k = \frac{(- r_1^{(k)} q_3^{(k)}+r_3^{(k)} q_1^{(k)})\alpha +( p_1^{(k)} q_3^{(k)}-p_3^{(k)} q_1^{(k)})\beta + (- p_1^{(k)}r_3^{(k)}+p_3^{(k)} r_1^{(k)})}{(r_1^{(k)} q_2^{(k)}-r_2^{(k)} q_1^{(k)})\alpha + (- p_1^{(k)} q_2^{(k)}+p_2^{(k)} q_1^{(k)})\beta + ( p_1^{(k)} r_2^{(k)}-p_2^{(k)} r_1^{(k)})}. 
\end{eqnarray}

\noindent From (4), we see that
\[
\alpha_k = \frac{(\alpha, \beta, 1) \theta_k (\textbf{e}_1) }{(\alpha, \beta, 1) \theta_k (\textbf{e}_3)},
\ \ \beta_k = \frac{(\alpha, \beta, 1) \theta_k (\textbf{e}_2)}{(\alpha, \beta, 1) \theta_k (\textbf{e}_3)}.
\]

\noindent Because \( \alpha_k >0, \beta_k>0 \) for \( k \ge 1\), it is enough to show that 
\[ (\alpha, \beta, 1) \theta_k (\textbf{e}_3) >0.\]
From lemma 3.5, we obtain that
\begin{eqnarray*}
(\alpha, \beta, 1) \theta_k (\textbf{e}_3) = - \delta_k q^{(k)}_{2} \left( \alpha +  \beta - \frac{p^{(k)}_2 +r^{(k)}_2 }{q^{(k)}_2  } \right)
\end{eqnarray*}
Fromm (4), we see that
\begin{eqnarray*}
\alpha_k + \beta_k &=& \frac{(\ve_1 \cdots \ve_k) q_3^{(k)} \alpha + (\ve_1 \cdots \ve_k) q_3^{(k)} \beta - (\ve_1 \cdots \ve_k)(p_3^{(k)}+r_3^{(k)})}{-(\ve_1 \cdots \ve_k) q_2^{(k)} \alpha - (\ve_1 \cdots \ve_k) q_2^{(k)} \beta + (\ve_1 \cdots \ve_k) (p_2^{(k)} + r_2^{(k)} )}\\[2mm]
&=& - \frac{q_3^{(k)}}{q_2^{(k)}} \frac{\left( \alpha + \beta \right)- \left( \frac{p_3^{(k)}+r_3^{(k)}}{q_3^{(k)}} \right)}{\left( \alpha + \beta  \right)-\left(\frac{p_2^{(k)}+r_2^{(k)}}{q_2^{(k)}} \right)}.
\end{eqnarray*}
From (3), the right hand side of the above equation is equal to 
\[ - \frac{q_3^{(k)}}{q_2^{(k)}} \frac{\left( \alpha + \beta \right)- \left( \frac{p_2^{(k)}+r_2^{(k)}}{q_2^{(k)}} \right) + \frac{\delta_k}{q_2^{(k)}q_3^{(k)}}}{\left( \alpha + \beta  \right)-\left(\frac{p_2^{(k)}+r_2^{(k)}}{q_2^{(k)}} \right)}
=  - \frac{q_3^{(k)}}{q_2^{(k)}} \left( 1+\frac{\delta_k}{q_2^{(k)}q_3^{(k)}} \frac{1}{\left( \alpha + \beta  \right)-\left(\frac{p_2^{(k)}+r_2^{(k)}}{q_2^{(k)}} \right)} \right).
\]
Then we have
\[
 \alpha +  \beta - \frac{p^{(k)}_2 +r^{(k)}_2 }{q^{(k)}_2 } = - \frac{\delta_k}{q_2^{(k)}} \left( \frac{1}{ q_2^{(k)}(\alpha_k + \beta_k) + q^{(k)}_3} \right).
\]
Then we see that
\begin{eqnarray*}
(\alpha, \beta, 1) \theta_k (\textbf{e}_3) = \frac{1}{q_2^{(k)} (\alpha_k + \beta_k) + q_3^{(k)} } >0.
\end{eqnarray*}
Therefore we obtain
\[
( \alpha, \beta, 1) \theta_k ( \textbf{e}_i ) >0 \quad (i=1,2,3).
\]
We finish to prove our main theorem.
\end{proof}

\noindent

\noindent
\begin{flushright}
\begin{tabular}{l}
Koshiro Ishimura,\\
Department of , Toyo University \\
2100, Kujirai, Kawagoe city, Saitama, JAPAN\\
e-mail: ishimura@toyo.jp 
\end{tabular}
\end{flushright}

\end{document}